\documentclass[12pt]{amsart}
\usepackage{amssymb}

\theoremstyle{plain}
\newtheorem{theorem}{Theorem}
\newtheorem{corollary}[theorem]{Corollary}
\newtheorem{lemma}[theorem]{Lemma}

\theoremstyle{definition}

\newtheorem{remark}[theorem]{Remark}

\newtheorem{example}[theorem]{Example}

\theoremstyle{remark}

\newcommand{\abs}[1]{\lvert#1\rvert}
\newcommand{\norm}[1]{\lVert#1\rVert}
\newcommand{\bigabs}[1]{\bigl\lvert#1\bigr\rvert}
\newcommand{\bignorm}[1]{\bigl\lVert#1\bigr\rVert}

\newcommand{\set}[1]{\bigl\{#1\bigr\}}
\renewcommand{\le}{\leqslant}
\renewcommand{\ge}{\geqslant}
\renewcommand{\mid}{\::\:}

\newcommand{\term}[1]{{\textit{\textbf{#1}}}}

\def\C{\mathcal C}
\def\F{\mathcal F}
\def\D{\mathcal D}
\def\K{\mathcal K}

\def\DC{\mathcal D_\C}
\def\Cepsilon{\mathcal C_{\varepsilon}}
\def\EC{\mathcal E_\C}

\begin{document}
\baselineskip 18pt

\title[A version of Lomonosov's theorem]{A version of Lomonosov's theorem 
for collections of positive operators}

\author[A.~I. Popov]{Alexey~I. Popov}
\author[V.~G. Troitsky]{Vladimir~G. Troitsky}
\address{
         }
\email{apopov@math.ualberta.ca}
\email{vtroitsky@math.ualberta.ca}


\date{\today. ~Draft.}

\begin{abstract}
  It is known that for every Banach space $X$ and every proper
  $WOT$-closed subalgebra $\mathcal A$ of $L(X)$, if $\mathcal A$
  contains a compact operator then it is not transitive. That is,
  there exist non-zero $x\in X$ and $f\in X^*$ such that $\langle
  f,Tx\rangle=0$ for all $T\in\mathcal A$. In the case of algebras of
  adjoint operators on a dual Banach space, V.~Lomonosov extended this
  as follows: without having a compact operator in the algebra, one
  has $\bigl\lvert\langle f,Tx\rangle\bigr\rvert\le\lVert T_*\rVert_e$
  for all $T\in\mathcal A$. In this paper, we prove a similar
  extension (in case of adjoint operators) of a result of R.~Drnov\v
  sek.  Namely, we prove that if $\mathcal C$ is a collection of
  positive adjoint operators on a Banach lattice $X$ satisfying
  certain conditions, then there exist non-zero $x\in X_+$ and $f\in
  X^*_+$ such that $\langle f,Tx\rangle\le\lVert T_*\rVert_e$ for all
  $T\in\mathcal C$.
\end{abstract}

\maketitle


In this paper we use techniques which were recently developed for
transitive algebras to obtain analogous results for collections of
positive operators on Banach lattices. Let us first briefly describe
these two branches of the Invariant Subspace research.

\subsection*{Transitive algebras} Suppose that $X$ is a Banach space.
A subspace $Z$ of $X$ is said to be invariant under an operator $T\in
L(X)$ if $\{0\}\ne Z\ne X$ and $T(Z)\subseteq Z$. The \emph{Invariant
  Subspace Problem} deals with the question: ``Which operators have
invariant subspaces?''.  Lomonosov proved in~\cite{Lomonosov:73} that
an operator which commutes with a compact operator has an invariant
subspace. There is also an algebraic version of the problem: which
subalgebras of $L(X)$ have no (common) invariant subspaces? Such
subalgebras are called \term{transitive}. The classical Burnside's
theorem asserts that if $X$ is finite-dimensional then $L(X)$ has no
proper transitive subalgebras (clearly, $L(X)$ itself is always
transitive). Using Lomonosov's technique, Burnside's theorem can be
extended to the infinite-dimensional case as follows:

\begin{theorem}[{\cite[Theorem 8.23]{Radjavi:73}}]\label{comp-alg}
  A proper $WOT$-closed subalgebra of $L(X)$ containing a compact
  operator is not transitive.
\end{theorem} 

A ``quantitative'' version of the later theorem was obtained by
Lomonosov in~\cite{Lomonosov:91} for algebras of adjoint operators.
Before we state it, we need to introduce some notation. It is easy to
see that a subalgebra $\mathcal A$ of $L(X)$ has an invariant subspace
if and only if there exist non-zero $x\in X$ and $f\in X^*$ such that
$\langle f,Tx\rangle=0$ for every $T\in\mathcal A$. Now suppose that
$X$ is a dual space; that is, $X=Y^*$ for some Banach space~$Y$.  If
$T\in L(X)$ is a bounded adjoint operator on $X$ then there is a
unique operator $S\in L(Y)$ such that $S^*=T$. We will write
$S=T_*$; there will be no ambiguity as $T_*$ will always be
taken with respect to~$Y$.  We will write $\norm{T}_e$ for the
essential norm of~$T$, i.e., the distance from $T$ to the space of
compact operators. Note that in general, for an adjoint operator~$T$, 
one has $\norm{T}_e\le\norm{T_*}_e$. See~\cite{Axler:80} for an 
example of $T$ such that $\norm{T}_e<\norm{T_*}_e$.

\begin{theorem}[\cite{Lomonosov:91}]\label{lom-dual-2}
  Let $X$ be a dual Banach space and $\mathcal A$ a proper
  $W^*OT$-closed subalgebra of $L(X)$ consisting of adjoint operators.
  Then there exist non-zero $x\in X$ and $f\in X^*$ such that
  $\bigabs{\langle f, Tx\rangle}\le\norm{T_*}_e$ for all $T\in\mathcal
  A$.
\end{theorem}

\subsection*{Invariant ideals of collections of positive operators}

Suppose now that $X$ is a Banach lattice. Recall that a linear (not
necessarily closed) subspace $\mathcal J\subseteq X$ is called an
\term{order ideal} if it is \term{solid}, i.e., $y\in \mathcal J$
implies $x\in \mathcal J$ whenever $\abs{x}\le\abs{y}$.  The following
version of Lomonosov's theorem for positive operators was proved by
B.~de~Pagter~\cite{dePagter:86}: \textit{a positive quasinilpotent
  compact operator on $X$ has a closed invariant order ideal}.  There
have been many extensions of this result, see,
e.g.,~\cite{Abramovich:02}. In particular, R.~Drnov\v
sek~\cite{Drnovsek:01} showed that a collection of positive operators
satisfying certain assumptions has a (common) invariant closed ideal.
To state his result precisely, we need to introduce more notations.

As usual, we write~$X_+$,~$X^*_+$, and $L(X)_+$ for the cones of
positive elements in~$X$,~$X^*$, and $L(X)$, respectively. 
Let $\C$ be a collection of positive operators on~$X$. 
Following~\cite{Abramovich:02}, we will denote by symbols
$\langle\C]$ and $[\C\rangle$ the \term{super left}
and the \term{super right commutants} of $\C$, respectively, i.e.,
\begin{equation*}
  \begin{split}
   \langle\C] &=
    \bigl\{S\in L(X)_+\colon ST\le TS\mbox{ for each }T\in\C\bigr\},\\
  [\C\rangle &=
    \bigl\{S\in L(X)_+\colon ST\ge TS\mbox{ for each }T\in\C\bigr\}.\\
  \end{split}
\end{equation*}
If $\D$ is another collection of operators then we write
$\C\D=\{TS\colon T\in\C, S\in \D\}$.
The symbol $\C^n$ is defined as the product of $n$ copies of $\C$.

An operator $T$ is \term{locally quasinilpotent} at $x$ if
$\limsup_n\norm{T^nx}^\frac{1}{n}=0$.  If $U$ is a subset of $X$ then
we write $\norm{U}=\sup\set{\norm{x}\colon x\in U}$.  We call a
collection $\C$ of operators \term{finitely quasinilpotent} at a
vector $x\in X$ if $\limsup_n\norm{\F^nx}^\frac{1}{n}=0$ for every
finite subcollection $\F$ of $\C$. Clearly, finite quasinilpotence at
$x$ implies local quasinilpotence at $x$ of every operator in the
collection.

If $E$ is a Banach lattice then an operator $T:E\to E$ is called
\term{AM-compact} if the image of every order interval under $T$ is
relatively compact.  Since order intervals are norm bounded, every
compact operator is AM-compact.  An operator $T$ is said to
\term{dominate} an operator $S$ if $\abs{Sx}\le T\abs{x}$ holds for all
$x\in E$.

\begin{theorem}[\cite{Drnovsek:01}] \label{drnovsek-modify} 
  If $\C$ is a collection of positive operators on a Banach lattice
  $X$ such that
  \begin{enumerate}
  \item $\C$ is finitely quasinilpotent at some positive non-zero
    vector, and
  \item some operator in $\C$ dominates a non-zero AM-compact
    operator,
  \end{enumerate}
  then $\C$ and $[\C\rangle$ have a closed invariant order ideal.
\end{theorem}

Observe that if a collection $\C$ of positive operators  has
a (closed nontrivial) invariant ideal then there exist non-zero
positive $x$ and $f$ such that $\langle f, Tx\rangle = 0$ for all
$T\in\C$.  The converse is also true when $\C$ is a semigroup.

The goal of this paper is to ``quantize'' Theorem~\ref{drnovsek-modify}
in the same manner that Theorem~\ref{comp-alg} was ``quantized'' into 
Theorem~\ref{lom-dual-2}. Our proofs use ideas from \cite{Lindstrom:00}
and~\cite{Michaels:77}.

In the following lemma, we collect several standard facts
that we will use later. See, e.g.,~\cite{Abramovich:02} for the
proofs.
\begin{lemma} \label{yudin} 
  Let $Z$ be a vector lattice, $x\in Z_+$.
  Then for each $y,z\in Z$ one has
  \begin{enumerate}
   \item\label{yu-1} $\abs{x\wedge y-x\wedge z}\le\abs{y-z}$;
   \item\label{yu-2} if $\abs{y}\le z$ then $\abs{x-x\wedge z}\le
    \abs{x-x\wedge y}$;
   \item\label{yu-3} $\abs{x-x\wedge y}\le\abs{x-y}$.
  \end{enumerate}
\end{lemma}

\bigskip

From now on, $X$ will be a real Banach lattice. We will also assume that $X$
is a dual Banach space; that is, $X=Y^*$ for some (fixed) Banach space $Y$.
We will start with a version of Theorem~\ref{lom-dual-2} for
convex collections of positive operators.

\begin{theorem}\label{pre-extension}
  Let $\C$ be a convex collection of positive adjoint operators
  on~$X$.  If there is $x_0>0$ such that every operator from $\C$ is
  locally quasinilpotent at $x_0$ then there exist non-zero $x\in X_+$
  and $f\in X^*_+$ such that $\langle f,Tx\rangle\le\norm{T_*}_e$
  for all $T\in\C$.
\end{theorem}

\begin{remark}
  One might try to deduce Theorem~\ref{pre-extension} from
  Theorem~\ref{lom-dual-2} by considering the $W^*OT$-closed algebra
  generated by~$\C$. However, the example in~\cite{Hadwin:86} shows
  that there exists an algebra of nilpotent operators on a Hilbert
  space $H$ which is $WOT$-dense in $L(H)$.
\end{remark}

\begin{proof}[Proof of Theorem~\ref{pre-extension}]
  Clearly, we may assume that $\norm{x_0}=1$.  Also, without loss of
  generality, $\C$ is closed under taking positive multiples of its
  elements, otherwise we replace $\C$ with $\set{\alpha T\colon
    T\in\C,0<\alpha\in\mathbb R}$. 
  Fix $0<\varepsilon<\frac{1}{10}$. Define
  \begin{eqnarray*}
    \Cepsilon&=&
    \bigl\{T\in\C\colon\norm{T_*}_e<\varepsilon\bigr\}\mbox{ and}\\
     H_\varepsilon(x)&=&\bigl\{z\in X\colon \abs{z}\le Tx
      \mbox{ for some }T\in\Cepsilon\bigr\},\quad x\in X_+.
  \end{eqnarray*}
  Then $H_\varepsilon(x)$ is convex
  and solid for all $x\in X_+$.

  Suppose that $\overline{H_\varepsilon(x)}\ne X$ for some nonzero
  $x\in X_+$. Since $H_\varepsilon(x)$ is convex, there
  is a nonzero $g\in X^*$ such that $g(y)\le 1$ for all $y\in
  H_\varepsilon(x)$.  Consider $h=\abs{g}\in X^*$.
  Then 
  for any $y\in H_\varepsilon(x)$ we have 
  \begin{displaymath}
    h(y)\le h\bigl(\abs{y}\bigr)=
      \sup\bigl\{g(u)\colon-\abs{y} \le u\le \abs{y}\bigr\}\le 1
  \end{displaymath}
  since $H_\varepsilon(x)$ is solid. In particular, $\langle
  h,Tx\rangle\le 1$ for all $T\in\Cepsilon$.

  Put $f=\frac{\varepsilon}{2}h$. We claim that 
  $\langle f,Tx\rangle\le\norm{T_*}_e$ for each $T\in\C$. 
  Indeed, if $T$ is compact,
  i.e., $\norm{T_*}_e=0$, then $\alpha T\in\Cepsilon$ for all
  $0<\alpha\in\mathbb R$. Therefore $\langle h,\alpha Tx\rangle\le 1$
  for all $0<\alpha\in\mathbb R$, so that $\langle f,Tx\rangle=\langle
  h,Tx\rangle=0$.  If $T$ is not compact then $\frac{\varepsilon
    T}{2\norm{T_*}_e}\in \Cepsilon$, whence
  \begin{displaymath}
    \langle f,Tx\rangle=\norm{T_*}_e\Big
    \langle h,\frac{\varepsilon T}{2\norm{T_*}_e}x\Big\rangle
    \le\norm{T_*}_e.
  \end{displaymath}

  Suppose now that $\overline{H_\varepsilon(x)}=X$ for all nonzero
  $x\in X_+$.  Then, in particular, for each $x\in X$ there is $y_x\in
  H_\varepsilon(x)$ such that $\norm{x_0-y_x}<\varepsilon$.  Fix an
  operator $T_x\in\Cepsilon$ such that $\abs{y_x}\le T_xx$.  Then
  \eqref{yu-2} and \eqref{yu-3} of Lemma~\ref{yudin} yield
  $\norm{x_0-x_0\wedge T_xx}<\varepsilon$.

  Let $U_0=\bigl\{x\in X_+\colon\norm{x-x_0}\le\frac{1}{2}\bigr\}$.
  Since $\bignorm{(T_x)_*}_e<\varepsilon$, there is an adjoint compact
  operator $K_x\in\K(X)$ such that $\norm{K_x-T_x}<\varepsilon$.  As
  compact adjoint operators are $w^*$-$\norm{\cdot}$ continuous on
  norm bounded sets, it follows that there is a relative (to $U_0$)
  $w^*$-\,open neigborhood $W_x\subseteq U_0$ of $x$ such that
  $\norm{K_xz-K_xx}<\varepsilon$ whenever $z\in W_x$.
  Then for every $y\in W_x$ we have:
  \begin{multline*}
   \bignorm{x_0-x_0\wedge T_xy}\le
   \bignorm{x_0-x_0\wedge T_xx}
   +\bignorm{x_0\wedge T_xx-x_0\wedge K_xx}\\
   +\bignorm{x_0\wedge K_xx-x_0\wedge K_xy}
   +\bignorm{x_0\wedge K_xy-x_0\wedge T_xy}\\
   \le\bignorm{x_0-x_0\wedge T_xx}+\norm{T_xx-K_xx}
   +\norm{K_xx-K_xy}+\norm{K_xy-T_xy}\\
   <\varepsilon+\varepsilon\norm{x}+\varepsilon
   +\varepsilon\norm{y}<5\varepsilon<\tfrac{1}{2}.
  \end{multline*}
  Together with $T_x\ge 0$ this yields $(x_0\wedge T_xy)\in U_0$ for each 
  $y\in W_x$.

  Note that $U_0$ is $w^*$-\,compact since $U_0$ is the intersection
  of $X_+$ with a closed ball.
  Hence, we can find $x_1,\dots,x_n\in U_0$ such that
  $U_0=\bigcup_{k=1}^n W_{x_k}$. Define
  $T=T_{x_1}+\dots+T_{x_n}\in\C$.  Then by
  Lemma~\ref{yudin}\eqref{yu-2}, we have $x_0\wedge Tx\in U_0$ for
  every $x\in U_0$.

  Define a sequence $(y_n)\subseteq U_0$ by $y_0=x_0$ and
  $y_{n+1}=x_0\wedge Ty_n$. Clearly $0\le y_n$ for all~$n$, and
  $y_{n}\le Ty_{n-1}\le\dots\le T^{n}y_0$, so that
  $\norm{y_n}\le\norm{T^nx_0}$. Thus $y_n\to 0$ as $n\to\infty$ by
  local quasinilpotence. This is a contradiction by the definition
  of~$U_0$.
\end{proof}

The next theorem shows that the conclusion of
Theorem~\ref{pre-extension} is also true for some collections of
operators which are not necessarily convex. We will, however, use a
more restrictive quasinilpotence condition. We will need some
additional definitions.

Let $\C$ be a collection of positive operators.
Following~\cite{Abramovich:02}, define
\begin{multline*}
  \DC=\Bigl\{D\in L(X)_+{}\mid\exists\space T_1,\dots,T_k
         \in[\C\rangle \mbox{ and }\\
  S_1,\dots,S_k\in\bigcup\limits_{n=1}^\infty\C^n
    \mbox{ such that }D\le\sum_{i=1}^kT_iS_i\Bigr\}
\end{multline*}
In other words, $\DC$ is the smallest additive and multiplicative
semigroup which contains the collection $[\C\rangle\cdot\C$ and such
that $T\in\DC$ and $0\le S\le T$ imply $S\in\DC$
(see~\cite{Abramovich:02}).

Let $\C$ be a collection of positive adjoint operators on~$X$. Define 
$$
\EC=\bigl\{T\in\DC\colon T=S^*\mbox{ for some }S\in L(Y)\bigr\}.
$$
Since adjoint operators are stable under addition and multiplication, 
$\EC$ is an additive and multiplicative semigroup. 
It is also clear that $\C\subseteq\EC$.

\begin{theorem}\label{extension}
  Let $\C$ be a collection of positive adjoint operators on~$X$.  If
  $\C$ is finitely quasinilpotent at some $x_0>0$
  then there exist non-zero $x\in X_+$ and $f\in X^*_+$ such that
  $\langle f,Tx\rangle\le\norm{T_*}_e$ for all $T\in\EC$.
\end{theorem}

\begin{proof}
  Clearly $\EC$ is convex.  Note that the finite quasinilpotence of
  $\C$ at $x_0$ implies the finite quasinilpotence of $\DC$ (and,
  therefore, of $\EC$) at $x_0$ (see,
  e.g.,~\cite[Lemma 10.4]{Abramovich:02}). Finally,
  apply Theorem~\ref{pre-extension} to~$\EC$.
\end{proof}



Now suppose, in addition, that $Y$ is itself a Banach lattice.
Then we can improve the conclusion of Theorem \ref{pre-extension}. 
For an operator $T$ acting on~$Y$, define
$$
\theta(T)=\inf\bigl\{\norm{T-K}\colon K\mbox{ is AM-compact}\bigr\}.
$$
Clearly, $\theta$ is a seminorm on $L(Y)$ and $\theta(T)=0$ if and
only if $T$ is AM-compact (because the subspace of AM-compact
operators in L(Y) is norm closed). 

For $\xi\in Y_+$, define a seminorm $\rho_\xi$ on $X$ via
$\rho_\xi(x)=\abs{x}(\xi)$.

\begin{lemma}\label{AM-continuity}
  If $\xi\in Y_+$ and $K\in L(Y)$ is AM-compact, then
  $K^*\colon(B_X,w^*)\to(X,\rho_{\xi})$ is continuous.
\end{lemma}

\begin{proof}
  Let $x_\alpha\xrightarrow{w^*}x$, with $x_\alpha$,
  $x\in B_X$. Write
  \begin{displaymath}
   \rho_{\xi}\bigl(K^*x_\alpha-K^*x\bigr)
   =\bigabs{K^*x_\alpha-K^*x}(\xi)
   =\sup\limits_{-\xi\le\zeta\le\xi}
   \langle x_\alpha-x,K\zeta\rangle
   =\sup\limits_{\nu\in A}\langle x_\alpha-x,\nu\rangle,
  \end{displaymath}
  where $A=K\bigl([-\xi,\xi]\bigr)$. By assumption, $K$ is
  AM-compact, thus $A$ is a $\norm{\cdot}$-compact set.

  For $\nu\in A$, fix $\alpha_\nu$ such that
  \begin{math}
   \bigabs{\langle x_\alpha-x,\nu\rangle}<\frac{\varepsilon}{3}
  \end{math}
  whenever $\alpha\ge\alpha_\nu$. If $\mu\in Y$ is such that
  $\norm{\mu-\nu}<\frac{\varepsilon}{3}$ then for
  $\alpha\ge\alpha_\nu$ we have
  \begin{displaymath}
   \bigabs{\langle x_\alpha-x,\mu\rangle}
   \le\tfrac{\varepsilon}{3}\norm{x_\alpha-x}+
   \bigabs{\langle x_\alpha-x,\nu\rangle}<
   \tfrac{2\varepsilon}{3}+\tfrac{\varepsilon}{3}
   =\varepsilon.
  \end{displaymath}
  Pick $\nu_1,\dots,\nu_n\in A$ such that
  $A\subseteq\bigcup\limits_{k=1}^n B(\nu_k,\frac{\varepsilon}{3})$.
  Then for every
  $\alpha\ge\max\set{\alpha_{\nu_1},\dots,\alpha_{\nu_n}}$ we must
  have $\rho_{\xi}(K^*x_\alpha-K^*x)<\varepsilon$.
\end{proof}

An operator $T\in L(X)$ will be said \term{$w^*$-locally 
quasinilpotent} at a pair $(x_0,\xi_0)$, where $x_0\in X$ and 
$\xi_0\in Y$, if $\bigabs{T^nx_0(\xi_0)}^\frac{1}{n}
\to 0$. Clearly, if $T$ is locally quasinilpotent at $x_0$ then 
$T$ is $w^*$-locally quasinilpotent at $(x_0,\xi_0)$ for every
$\xi_0\in Y$.

\begin{theorem}\label{pre-strong-ext}
  Suppose that $X=Y^*$ for some Banach lattice $Y$, and $\C$ is a convex
  collection of positive adjoint operators on~$X$.  Suppose that there
  exists a pair $(x_0,\xi_0)\in X_+\times Y_+$ such that $x_0(\xi_0)\ne 0$
  and every operator from $\C$ is $w^*$-locally quasinilpotent at
  $(x_0,\xi_0)$. Then there exist non-zero $x\in X_+$ and $f\in X^*_+$ such
  that $\langle f,Tx\rangle\le\theta(T_*)$ for all $T\in\C$.
\end{theorem}

\begin{proof}
  The proof of the theorem is similar to that of
  Theorem~\ref{pre-extension}. We may assume that $\norm{x_0}=1$,
  $\norm{\xi_0}=1$, and $\C$ is closed under
  taking positive multiples. Put
  $\rho_{\xi_0}(x)=\abs{x}(\xi_0)$. Evidently 
  $\rho_{\xi_0}(x)\le\norm{x}$ for all $x\in X$.
  It is also clear that $\abs{x}\le\abs{y}$ implies
  $\rho_{\xi_0}(x)\le\rho_{\xi_0}(y)$.

  Fix $0<\varepsilon<\frac{x_0(\xi_0)}{8}$. Define
  \begin{eqnarray*}
    \Cepsilon &=& \set{T\in\C\colon\theta(T_*)<\varepsilon}\mbox{ and}\\
    G_\varepsilon(x) &=& \set{z\in X\colon\abs{z}\le Tx\mbox{ for some }
      T\in\Cepsilon},\quad x\in X_+.
  \end{eqnarray*}
  Suppose that $G_\varepsilon(x)$ is not dense in $X$ for some $x\in
  X_+$. Analogously to the proof of Theorem~\ref{pre-extension}, we
  find a positive functional $h\in X^*_+$ such that $\langle
  h,Tx\rangle\le 1$ for all $T\in\Cepsilon$. Considering separately
  the cases $\theta(T_*)=0$ and $\theta(T_*)\ne 0$, we get the
  conclusion of the theorem.

  Thus, we may assume that $\overline{G_\varepsilon(x)}=X$ for all
  $x>0$. Define 
  \begin{displaymath}
  U_0=\set{x\in X_+\colon\norm{x}\le 1\mbox{ and
    }\rho_{\xi_0}(x-x_0)\le\frac{x_0(\xi_0)}{2}}.
  \end{displaymath}
  Clearly, $U_0$ is $w^*$-compact.

  Let $x\in U_0$ be arbitrary. Since $\overline{G_\varepsilon(x)}=X$,
  we can find $T_x\in\Cepsilon$ such that $\rho_{\xi_0}(x_0-x_0\wedge
  T_xx)\le\norm{x_0-x_0\wedge T_xx}<\varepsilon$.  Fix an operator
  $K_x$ adjoint to an AM-compact operator such that
  $\norm{T_x-K_x}<\varepsilon$.  By Lemma~\ref{AM-continuity}, we can
  find a relative (to $U_0$) $w^*$-open neighborhood $V_x\subseteq U_0$ 
  of $x$   such that $\rho_{\xi_0}(K_xx-K_xz)<\varepsilon$ for all 
  $z\in V_x$. Then for an arbitrary $z\in V_x$, we have
  \begin{multline*}
    \rho_{\xi_0}\bigl(x_0-x_0\wedge T_xz\bigr)
    \le\rho_{\xi_0}\bigl(x_0-x_0\wedge T_xx\bigr)
    +\rho_{\xi_0}\bigl(x_0\wedge T_xx-x_0\wedge K_xx\bigr)\\    
    +\rho_{\xi_0}\bigl(x_0\wedge K_xx-x_0\wedge K_xz\bigr)
    +\rho_{\xi_0}\bigl(x_0\wedge K_xz-x_0\wedge T_xz\bigr)\\
    <\varepsilon+\norm{T_xx-K_xx}
    +\rho_{\xi_0}(K_xx-K_xz)+\norm{K_xz-T_xz}\\
    <\varepsilon+\norm{T_x-K_x}\cdot\norm{x}+\varepsilon
    +\norm{T_x-K_x}\cdot\norm{z}<4\varepsilon<\frac{x_0(\xi_0)}{2}.
  \end{multline*}
  Take $x_1,\dots,x_m$ in $U_0$ such that
  $\bigcup\limits_{k=1}^mV_{x_k}=U_0$. Then
  $T=T_{x_1}+\dots+T_{x_k}\in\C$ satisfies $\rho_{\xi_0}\bigl(x_0-
  x_0\wedge Tz\bigr)\le\frac{x_0(\xi_0)}{2}$ for all $z\in U_0$.  
  Since   $\norm{x_0\wedge Tz}\le\norm{x_0}=1$, we have 
  $x_0\wedge Tz\in U_0$ for all $z\in U_0$.

  Put $z_0=x_0$ and $z_{n+1}=x_0\wedge Tz_n$. By the $w^*$-local
  quasinilpotence of $T$ at $(x_0,\xi_0)$ we have
  $\rho_{\xi_0}(z_n)\le\rho_{\xi_0}(T^nx_0)=\bigabs{T^nx_0(\xi_0)}\to 0$ 
  as $n\to\infty$ which is impossible by the definition of~$U_0$.
\end{proof}

The following result is derived from Theorem~\ref{pre-strong-ext}
in the same way that Theorem~\ref{pre-extension} was deduced from
Theorem~\ref{extension}.

\begin{theorem}\label{strong-ext}
  Suppose that $X=Y^*$ for some Banach lattice $Y$, and $\C$ is a
  collection of positive adjoint operators on~$X$. If $\C$ is finitely
  quasinilpotent at some $x_0>0$ then there exist non-zero $x\in X_+$
  and $f\in X^*_+$ such that $\langle f,Tx\rangle\le\theta(T_*)$ for
  all $T\in\EC$.
\end{theorem}

As every operator on $\ell_p$ $(1\le p<\infty)$ is AM-compact, 
this theorem can be used as an alternative proof of the following 
(certainly known) result.

\begin{corollary}\label{ellp}
  Every collection of positive operators on $\ell_p$, $1< p<\infty$,
  which is finitely quasinilpotent at a non-zero positive vector, has a
  non-trivial closed common invariant ideal.
\end{corollary}

Of course, Corollary~\ref{ellp} follows easily from
Theorem~\ref{drnovsek-modify} when $1\le p<\infty$.

\begin{corollary}
  Every collection of positive adjoint operators on $\ell_\infty$
  which is finitely quasinilpotent at a non-zero positive vector has a
  non-trivial closed common invariant ideal.
\end{corollary}

The following example shows that the assumptions in
Theorems~\ref{extension} and~\ref{strong-ext} in general do not guarantee
the existence of an invariant subspace.

\begin{example}\label{example}
  {\it There is a collection $\C$ of operators which satisfies all the
  conditions of Theorem~\ref{strong-ext} and has no common non-trivial
  invariant subspaces.} Namely, in~\cite{Drnovsek:02}, the authors 
  constructed a multiplicative semigroup $\mathcal S_p$ of positive 
  square-zero operators acting on $L_p[0,1]$, $1\le p<\infty$, having 
  no common non-trivial invariant subspaces. It is not difficult to show 
  that $\mathcal S_p$ is in fact finitely quasinilpotent at every positive 
  vector. Hence for $1<p<\infty$, $\C=\mathcal S_p$ satisfies the 
  conditions of Theorem~\ref{strong-ext}.
\end{example}

\begin{remark}
  Even though Theorem~\ref{comp-alg} is not a special case of
  Theorem~\ref{lom-dual-2}, in the case of an algebra of adjoint
  operators the former can be easily deduced from the latter,
  see~\cite[Corollary 1]{Lomonosov:91}. Similarly, we will show that
  in case of adjoint operators, Theorem~\ref{drnovsek-modify} can be
  deduced from Theorem~\ref{strong-ext}. Indeed, suppose that $X=Y^*$
  for some Banach lattice~$Y$, and $\C$ is a collection of positive
  adjoint operators which is finitely quasinilpotent at some $x_0>0$
  and some operator in it dominates a non-zero AM-compact
  positive\footnote{Unlike in Theorem~\ref{drnovsek-modify}, we also
    require that $K\ge 0$ here.} adjoint operator~$K$.  We will show
  that there is a non-trivial closed ideal which is invariant under
  $\C$ and under all adjoint operators in~$[\C\rangle$.

  Clearly, $K\in\EC$. Let $x$ and $f$ be as in Theorem~\ref{strong-ext}.
  \begin{eqnarray*}
    \mathcal J_1 &=&
      \bigl\{z\in X\mid\abs{z}\le T_1KT_2x\mbox{ for some }T_1,T_2\in\EC\bigr\},\\
    \mathcal J_2 &=&
      \bigl\{z\in X\mid T\abs{z}=0\mbox{ for all }T\in\EC\bigr\},\mbox{ and}\\
    \mathcal J_3 &=&
      \bigl\{z\in X\mid\abs{z}\le Tx\mbox{ for some }T\in\EC\bigr\}.
  \end{eqnarray*}
  It is easy to see that $\mathcal J_1$, $\mathcal J_2$, and $\mathcal
  J_3$ are ideals in~$X$, invariant under $\C$ and under all adjoint
  operators in~$[\C\rangle$. It is left to show that at least one of
  the three must be non-trivial. Clearly, $\mathcal J_2$ is closed and
  $\mathcal J_2\ne X$. Suppose that $\mathcal J_2=\{0\}$. In
  particular, $x\notin J_2$. It follows that $J_3\ne\{0\}$. Suppose
  that $\mathcal J_3$ is dense in~$X$. It follows from
  Theorem~\ref{strong-ext} that $\mathcal J_1\subseteq\ker f$; hence
  $\overline{\mathcal J_1}$ is proper. Assume that $\mathcal
  J_1=\{0\}$. Hence, $T_1KT_2x=0$ for all $T_1,T_2\in\EC$. Since
  $\mathcal J_2=\{0\}$, it follows that $K$ vanishes on $\EC x$ and,
  therefore, on~$\mathcal J_3$.  Since $\mathcal J_3$ is dense in $X$
  it follows that $K=0$; a contradiction.
\end{remark}

\medskip

\textbf{Acknowledgment.} We would like to thank Victor Lomonosov for
helpful discussions.


\begin{thebibliography}{00}

\bibitem{Abramovich:02}
  Y.A.~Abramovich, C.D.~Aliprantis,
  \newblock \emph{An Invatation to Operator Theory},
  \newblock Graduate studies in mathematics, v.50.

\bibitem{Axler:80}
  S.~Axler, N.~Jewell, A.~Shields,
  \newblock The essential norm of an operator and its adjoint,
  \newblock \emph{Trans.\ Amer.\ Math.\ Soc.} 261 (1980), no.~1, 159--167. 

\bibitem{Drnovsek:01}
  R.~Drnov\v sek,
  \newblock Common Invariant Subspaces for Collections of Operators,
  \newblock \emph{Integral Eq.\ Oper.\ Th.} 39 (2001), 253--266.

\bibitem{Drnovsek:02}
  R.~Drnov\v sek, D.~Kokol-Bukov\v sek, L.~Livshits, G.~Macdonald,
  M.~Omladi\v c, and H.~Radjavi,
  \newblock An irreducible semigroup of non-negative square-zero operators,
  \newblock \emph{Integral Eq.\ Oper.\ Th.} 42 (2002), no.~4, 449--460.

\bibitem{Hadwin:86}
  D.~Hadwin, E.~Nordgren, M.~Radjabalipour, H.~Radjavi,
  and P.~Rosenthal,
  \newblock A nil algebra of bounded operators on Hilbert 
  space with semisimple norm closure,
  \newblock \emph{Integral Eq.\ Oper.\ Th.}  9 (1986), no.~5, 739--743.

\bibitem{Lindstrom:00}
  M.~Lindstr\"om and G.~Schl\"uchtermann,
  \newblock Lomonosov's techniques and Burnside's theorem,
  \newblock \emph{Canad.\ Math.\ Bull.}
  vol {\bf 43} (1), 2000 pp. 87-89.

\bibitem{Lomonosov:73}
  V.~Lomonosov,
  \newblock Invariant subspaces of the family of operators that
  commute with a completely continuous operator,
  \newblock \emph{Funktsional.\ Anal. i Prilozen}
  7 (1973), no.~3, 55-56 (Russian).

\bibitem{Lomonosov:91}
  V.~Lomonosov,
  \newblock An Extension of Burnside's Theorem to Infinite-Dimensional
  Spaces,
  \newblock \emph{Israel J.\ Math}, 75 (1991), 329-339.

\bibitem{Michaels:77}
  A.J.~Michaels,
  \newblock Hilden's simple proof of Lomonosov's invariant subspace
  theorem,
  \newblock \emph{Adv. in Math.} 25 (1977), 56-58.

\bibitem{dePagter:86}
  B.~de Pagter,
  \newblock Irreducible compact operators, 
  \newblock \emph{Math Z.} 192 (1986), 149-153.

\bibitem{Radjavi:73}
  H. Radjavi and P. Rosenthal,
  \newblock \emph{Invariant subspaces},
  \newblock Springer-Verlag, New York, 1973.

\end{thebibliography}
\end{document}